\documentclass[a4paper,11pt,reqno]{amsart}

\usepackage{amsmath,amssymb}
\usepackage{enumerate}
\usepackage{ifthen}
\usepackage{graphicx}
\usepackage{tikz}
\usepackage[hypcap]{caption}
\usetikzlibrary{arrows}
\tikzstyle{block}=[draw opacity=0.7,line width=1.4cm]
\nonstopmode\numberwithin{equation}{section}
\usepackage{fancyhdr,lipsum}
\newtheorem{theorem}{Theorem}[section]

\newtheorem{lemma}{Lemma}[section]
\newtheorem{example}{Example}[section]

\allowdisplaybreaks

 \theoremstyle{plain}

\newcommand{\Mod}[1]{\ (\mathrm{mod}\ #1)}

\begin{document}
\title{ Non trivial idempotents of the matrix rings over  polynomial ring $Z_{pqr}[x]$}
\maketitle
\begin{center}
\center{  $\text{Gaurav Mittal}$}
{\footnotesize
 % please put the address of the second  and third author
 \center{ Department of Mathematics, Indian Institute of Technology Roorkee, Roorkee, India} \\ email: gmittal@ma.iitr.ac.in
}

\medskip
\begin{abstract}   
In this paper, we study the  non trivial idempotents of the $2 \times 2$ matrix ring over the  polynomial ring  $\mathbb{Z}_{pqr}[x]$ for distinct primes $p, q $ and $r$ greater than $3$. We have classified all the idempotents of this matrix ring into several classes such that any idempotent must belong to one of these classes. This work is extension of the work done in $[1]$.
\end{abstract}
\end{center}
\hspace{-5mm}\subjclass {}{\textbf{Mathematics Subject Classification (2010)}: $16S50$, $13F20$}\\

\section{Introduction}
 
 In ring theory, an  element $a$ of a ring $R$ is idempotent if $a = a^2$. Then by induction, we have $a = a^2= \cdots = a^m$ for any positive integer $m$. Therefore, we can say that these elements resist to change on multiplying with themselves. With these idempotents, we can define several other classes of elements, for example, unit regular elements $[2]$, clean and strongly clean elements  $[5, 6, 7]$ and  Lie regular elements $[8]$ etc. and therefore,  idempotents are of interest among many researchers.

Now, let us discuss some of the available literature   in this direction.  In the case of polynomial ring $R[x]$, when $R$ is an abelian ring (a ring in which all idempotents are central), idempotents  of $R[x]$  and $R$ coincides $[4$, Lemma $1]$. For the  matrix rings over the  polynomial rings very few results are known.  Kanwar et al.  $[3]$ studied the idempotents in $M_2(\mathbb{Z}_{2p}[x])$ for odd prime $p$ and in $M_2(\mathbb{Z}_{3p}[x])$ for prime $p > 3$. Recently, Balmaceda et al. $[1]$ generalized the results of $[3]$ by determining the non trivial idempotents in the matrix rings $M_2(\mathbb{Z}_{pq}[x])$ for any primes $p$ and $q$.

In this paper, motivated by the importance of idempotents, we continue in this direction and study the idempotents of $2 \times 2$ matrices over the polynomial ring  $\mathbb{Z}_{pqr}[x]$ for primes  $p, q$ and $r$  greater than $3$. The special thing about  polynomial ring $\mathbb{Z}_{pqr}[x]$ is that  it is a  reduced ring (ring in which the zero element is the only nilpotent element). This paper is structured in the following manner: All the  basic results required in our work are accumulated in Section $2$. Our main result related to the idempotents is discussed in Section $3$. Section $4$ concludes the paper.

\section{Preliminaries}
We start this section by the following characterization of the reduced ring. The simple proof is given here for completeness.
\begin{lemma}
The ring   $\mathbb{Z}_n$  is a reduced ring,  if $n$ has no square factor in it.  
\end{lemma}
\begin{proof}

Let if possible $g$ be some non zero nilpotent element of $\mathbb{Z}_n$. Clearly $g^{m}$ is divisible by $n$ for some positive integer $m$. If  $n$ has no square factor in it, then each prime divisor of $n$ also divides $g^m$ and hence each of them divides $g$. Since every prime divisor of $n$ divides $g$ which means $n$ is a divisor of $g$ which is absurd as $g < n$. Hence result.  
 \end{proof}

\begin{theorem}
For the ring $Z_n$, when $n$ is square free,  the determinant and the trace of  every idempotent in $M_2(\mathbb{Z}_n[x])$   is in $\mathbb{Z}_{n}$.
\end{theorem}
\begin{proof}
Follows directly from  Lemma $2.1$  and $[3,$ Proposition $3.1]$.
\end{proof}

\begin{theorem}
If $R$ is a reduced ring, then idempotents of $R[x]$ and $R$ coincide.
\end{theorem}

\begin{proof}
$[3$, Corollary $2.4]$. 
\end{proof}
Next result is  about the number of idempotents in $\mathbb{Z}_n$. We skip its proof as it is not hard to see.
\begin{lemma}
For the ring $\mathbb{Z}_n$, where $n$ is a  positive integer, total number of idempotents are $2^m$, where $m$ is the number of distinct prime  divisors of $n$. 
\end{lemma}

\section{Idempotents of matrix ring $M_2(\mathbb{Z}_{pqr}[x])$}

In this section, we discuss our main result in the form of idempotents of the ring  of $2 \times 2$ matrices over the polynomial ring $\mathbb{Z}_{pqr}[x]$, where $p, q$ and $r$ are distinct primes greater than $ 3$.  To start with, let us first
 discuss the following three lemmas which would be helpful in proving our main result. The very first lemma is a result on the  idempotents of $\mathbb{Z}_{pqr}[x]$.
\begin{lemma}

 Idempotents of the  polynomial ring $\mathbb{Z}_{pqr}[x]$, where  $p, q, r$ are distinct primes, are  $$ 0, \ 1,\ (pq)^{r-1},  \ (pr)^{q-1},\ (qr)^{p-1}, \ p^{(q-1)(r-1)},\ q^{(p-1)(r-1)},\ r^{(p-1)(q-1)}.$$

\end{lemma}

\begin{proof}
From Theorem $2.2$ and Lemma $2.2$,  it follows that the ring $\mathbb{Z}_{pqr}$ has $8$ idempotents. Now, let $y$ be any  idempotent of $\mathbb{Z}_{pqr}$. This means $y^2 \equiv y \Mod{pqr}$. Solving  this congruence is equivalent to solve the following system of congruences $$y^2 \equiv y \Mod{p}, \quad y^2 \equiv y \Mod{q}, \quad y^2 \equiv y \Mod{r}.$$ Each of these congruences has $2$ solutions which are given by $y \equiv 0$ or $y \equiv 1$. So the congruence $y^2 \equiv y \Mod{pqr}$ has $8$ solutions   discussed in the following cases. 

Case-$1$:   $y \equiv 0 \Mod{p},\ \ y \equiv 0 \Mod{q}, \  \  y \equiv 0 \Mod{r}$.
Here, clearly $y \equiv 0 \Mod{pqr}$.

Case-$2$:   $y \equiv 0 \Mod{p},\ \ y \equiv 0 \Mod{q}, \  \  y \equiv 1 \Mod{r}$. This means
 $y \equiv 0 \Mod{pq}$ and $ y \equiv 1 \Mod{r}$. Further, $y \equiv 0 \Mod{pq}$ implies $y = pq K$ for some $K \in \mathbb{Z}$. Putting this value of $y$ in $ y \equiv 1 \Mod{r}$, we get $pqK \equiv 1 \Mod{r}$. From this congurence, on employing Euler's equation,  we get a solution  given by $K = (pq)^{r-2}$. Thus, $y \equiv (pq)^{r-1} \Mod{pqr}$.

Case-$3$:   $y \equiv 0 \Mod{p},\ \ y \equiv 1 \Mod{q}, \  \  y \equiv 0 \Mod{r}$. On the similar lines of Case $2$, we can easily see that  $y \equiv (pr)^{r-1} \Mod{pqr}$.

Case-$4$:    $y \equiv 1 \Mod{p},\ \ y \equiv 0 \Mod{q}, \  \  y \equiv 0 \Mod{r}$. Here  $y \equiv (qr)^{r-1} \Mod{pqr}$.

Case-$5$: If   $y \equiv 0 \Mod{p},\ \ y \equiv 1 \Mod{q}, \  \  y \equiv 1 \Mod{r}$.
Then, clearly $ y \equiv 1 \Mod{qr}$. The congruence $y \equiv 0 \Mod{p}$ implies that $y = p K$ for some $K \in \mathbb{Z}$. Substituting the value of $y$ in $ y \equiv 1 \Mod{qr}$, we get $pK \equiv 1 \Mod{qr}$. From this congurence, on employing Euler's theorem,  we get a solution  given by $K = p^{(q-1)(r-1)-1}$. Thus, $y \equiv p^{(q-1)(r-1)} \Mod{pqr}$.

Case-$6$: If   $y \equiv 1 \Mod{p},\ \ y \equiv 0 \Mod{q}, \  \  y \equiv 1 \Mod{r}$.
As done in Case $5$, we can see that $y \equiv q^{(p-1)(r-1)} \Mod{pqr}$.

Case-$7$: If   $y \equiv 1 \Mod{p},\ \ y \equiv 1 \Mod{q}, \  \  y \equiv 0 \Mod{r}$.
Here,  $y \equiv r^{(p-1)(q-1)} \Mod{pqr}$.

Case-$8$: If $y \equiv 1 \Mod{p},\ \ y \equiv 1 \Mod{q}, \  \  y \equiv 1 \Mod{r}$. Clearly,  $y \equiv 1 \Mod{pqr}$. 
\end{proof}
\begin{example} For $\mathbb{Z}_{105} = \mathbb{Z}_{3 \times 5 \times 7}$, $p = 3, q= 5, r = 7$, idempotents  are $0, \ 1,\ 15^6 \equiv 15,\ 21^4 \equiv 21, \ 35^2 \equiv 70,\ 3^{24} \equiv 36,\ 5^{12} = 85,\ 7^{8} = 91$.   

\end{example}

\textbf{Notation}-  Now onwards, if we write $G = \begin{bmatrix} e & f \\ g & h \end{bmatrix}$, then we always assume $e, f, g, h$ are polynomials in $\mathbb{Z}_{pqr}[x]$. \newline
In the following lemma, we discuss about the solutions of the quadratic congruence $x^2   \equiv x+2p^{(q-1)(r-1)} \Mod{pqr}$.
 \begin{lemma}
 For the quadratic congruence $x^2   \equiv x+2p^{(q-1)(r-1)} \Mod{pqr}$, where $p, q, r$ are distinct odd primes, solutions  modulo $pqr$ are $$\bigg\{2p^{(q-1)(r-1)}, \ p^{(q-1)(r-1)}+1, \ -p^{(q-1)(r-1)},\ 1-2p^{(q-1)(r-1)}, \  \big((-1-2p^{q-1}\big)(pq)^{r-1} + 2 p^{q-1}), $$  $$ \big((-2-p^{q-1})(pq)^{r-1}+p^{q-1}+1\big),  \ ((-1-2p^{r-1})(pr)^{q-1}+2p^{r-1}), \ \big((-2-p^{r-1})(pr)^{q-1}+p^{r-1}+1\big)\bigg\}. $$  

\end{lemma}
\begin{proof}
Since $\gcd (p,q, r) = 1$, the congruence  $x^2   \equiv x+2p^{(q-1)(r-1)} \Mod{pqr}$ has solution if and only if the system of congruences  $x^2   \equiv x+2p^{(q-1)(r-1)} \equiv x \Mod{p}$,    $x^2   \equiv x+2p^{(q-1)(r-1)}  \Mod{q}$ and $x^2   \equiv x+2p^{(q-1)(r-1)} \Mod{r}$ has a solution. Let's name these congruencs as $(a)$,  $(b)$ and $(c)$,  respectively. Observe that $(b)$ is also equivalent to $x^2 \equiv x+2 \Mod{q}$ using Euler's theorem. Similarly, $(c)$ is equivalent to $x^2 \equiv x+2 \Mod{r}$.    So for solving the given congruence, we have to solve the following equations $$ x^2\equiv x  \Mod{p}, \quad x^2\equiv x+2 \Mod{q} \quad \text{and} \quad   x^2 \equiv x+2 \Mod{r}.$$ Further, note that both  the congruences  $x^2\equiv x+2 \Mod{q}$ and $x^2\equiv x+2 \Mod{r}$ have two solutions, given by $x = 2$ or $x = -1$. Similarly,  the congruence $x^2\equiv x \Mod{p}$ has two solutions, given by $x = 0$ or $x = 1$. We consider all these possibilities and discuss them in the following eight  cases.  \\
Case-1: When $x \equiv  0\Mod{p}$, \  $x \equiv  2\Mod{q}$ and $x \equiv  2\Mod{r}$. On solving these, we get $x \equiv  2p^{(q-1)(r-1)}\Mod{pqr}$. \\
Case-2: When $x \equiv  1\Mod{p}$, \  $x \equiv  2\Mod{q}$ and $x \equiv  2\Mod{r}$. Here we get,  $x \equiv  1+p^{(q-1)(r-1)}\Mod{pqr}$.\\
Case-3: When $x \equiv  0\Mod{p}$, \  $x \equiv  -1\Mod{q}$ and $x \equiv  -1\Mod{r}$. On solving these, we get  $x \equiv  -p^{(q-1)(r-1)}\Mod{pqr}$.  \\
Case-4: When $x \equiv  1\Mod{p}$, \  $x \equiv  -1\Mod{q}$ and $x \equiv  -1\Mod{r}$. Here we get,  $x \equiv  1-2p^{(q-1)(r-1)}\Mod{pqr}$.\\
Case-5: When $x \equiv  0\Mod{p}$, \  $x \equiv  2\Mod{q}$ and $x \equiv  -1\Mod{r}$. Here we get,  $x \equiv  (-1-2p^{q-1}\big)(pq)^{r-1} + 2 p^{q-1}\Mod{pqr}$.\\
Case-6:  When $x \equiv  1\Mod{p}$, \  $x \equiv  2\Mod{q}$ and $x \equiv  -1\Mod{r}$. On solving these, we get $x \equiv  (-2-p^{q-1}\big)(pq)^{r-1} +  p^{q-1} + 1\Mod{pqr}$. \\
Case-7: When $x \equiv  0\Mod{p}$, \  $x \equiv  -1\Mod{q}$ and $x \equiv  2\Mod{r}$. In this case,  $x \equiv  (-1-2p^{r-1}\big)(pr)^{q-1} + 2 p^{r-1} \Mod{pqr}$.\\ 
Case-8:  When $x \equiv  1\Mod{p}$, \  $x \equiv  -1\Mod{q}$ and $x \equiv  2\Mod{r}$. Here  $x \equiv  (-2-p^{r-1}\big)(pr)^{q-1} +  p^{r-1} + 1\Mod{pqr}$.  
\end{proof}
In the following lemma, we discuss about the solutions of the quadratic congruence $x^2   \equiv x+2(pq)^{r-1} \Mod{pqr}$.
\begin{lemma}
For the quadratic congruence $x^2   \equiv x+2(pq)^{r-1} \Mod{pqr}$, where $p, q, r$ are distinct odd primes, solutions  modulo $pqr$ are $$\bigg\{2(pq)^{r-1}, \ -(pq)^{r-1}, \ (pq)^{r-1} + 1 ,\ 1-2(pq)^{r-1}, \ (2-p^{q-1}\big)(pq)^{r-1} + p^{q-1}, $$  $$ (-1-p^{q-1}\big)(pq)^{r-1}+p^{q-1}, \ (2-q^{p-1})(pq)^{r-1} + q^{p-1}, \ (-1-q^{p-1})(pq)^{r-1} + q^{p-1}\bigg\}. $$  

\end{lemma}
\begin{proof}
This can be proved on the similar lines of Lemma $3.2$. 
\end{proof}

Now, we are ready to give our  main result of the paper in which we classify all the  idempotents of $M_2(\mathbb{Z}_{pqr}[x])$. 

\begin{theorem}
  Any non trivial idempotent of the matrix ring $M_2$ over the polynomial ring $\mathbb{Z}_{pqr}[x]$  for distinct  primes $p, q, r $ greater than $3$, is one of the  following forms:
 
 \begin{enumerate} \item If $\det G = 0$, then we have the following possibilities:

 \begin{enumerate}
 \item  $ G = \begin{bmatrix}
 e(x) & f(x) \\ g(x) & 1-e(x)
 \end{bmatrix}$, where $e(x)(1-e(x))- g(x)f(x) = 0$. 

 \item  $ G = \begin{bmatrix}
 Ie(x) & I f(x) \\ I g(x) & I( 1-e(x))
 \end{bmatrix}$, where $e(x)(1-e(x))- g(x)f(x) = J k(x)$ for $k(x) \in \mathbb{Z}_{pqr}[x]$, where $I \in \{(pq)^{r-1},  (pr)^{q-1}, (qr)^{p-1}, p^{(q-1)(r-1)},  q^{(p-1)(r-1)}, r^{(p-1)(q-1)}\} $ and $J \in \{r,  q, p, qr,  pr, pq\}$. Here, value of $J$ at $i^{th}$ position in its set of possibilities corresponds to the value of $I$ at $i^{th}$ position in its respective set for $1 \leq i \leq 6$. 
  
 \end{enumerate}

 \item If $\det G = (pq)^{r-1}$, then we have the following four possibilities:
 
 \begin{enumerate}
 \item  $ G = \begin{bmatrix}
 (pq)^{r-1} & 0 \\ 0 & (pq)^{r-1} 
 \end{bmatrix}$. 

 \item  $ G = \begin{bmatrix}
 1+re(x) & rf(x) \\ rg(x) & (pq)^{r-1}  - re(x)
 \end{bmatrix}$, where $e(x)(1+re(x))+ rg(x)f(x) = pq k(x)$ for $k(x) \in \mathbb{Z}_{pqr}[x]$. 
 
 \item   $G = \begin{bmatrix}
u+pre(x) & prf(x) \\ pr g(x) & t-(u+pre(x))
\end{bmatrix}$, where $(u+pre(x))t-(u+pre(x))^2-(pr)^2f(x)g(x) \equiv (pq)^{r-1} \Mod{pqr}$,  $u \equiv 0\Mod{p}$ and $u \equiv 1 \Mod{r}$. Here $t = (2-p^{q-1})(pq)^{r-1}+p^{q-1}$.  
 
 \item By interchanging  the role of $p$ and $q$ in part $2(c)$, we get a new class of idempotents.  
 \end{enumerate}
 Similarly, if $\det G = (qr)^{p-1}$ or $\det G = (pr)^{q-1}$, we have $4$ possibilities in  each case.
 \item If $\det G = p^{(q-1)(r-1)}$, then we have the following $2$ possibilities:
 \begin{enumerate}
 \item  $ G = \begin{bmatrix}
 p^{(q-1)(r-1)} & 0 \\ 0 & p^{(q-1)(r-1)} 
 \end{bmatrix}$.
 
 \item  $ G = \begin{bmatrix}
 1+qre(x) & qrf(x) \\ qrg(x) & p^{(q-1)(r-1)}-qre(x)
 \end{bmatrix}$, where $e(x)(1+qre(x))+qrf(x)g(x) = p k(x)$ for $k(x) \in \mathbb{Z}_{pqr}[x]$.
 
 \end{enumerate}
 Similarly, if $\det G = q^{(p-1)(r-1)}$ or $\det G = r^{(p-1)(q-1)}$, we have $2$ possibilities in  each case.  
 
\end{enumerate}   
\end{theorem}

\begin{proof}
Let $G  = \begin{bmatrix}
e(x) & f(x) \\ g(x) & h(x)
\end{bmatrix}$ be a non trivial idempotent of $M_2(\mathbb{Z}_{pqr}[x])$. We write $e, f, g, h$ in place of $e(x), f(x), g(x), h(x)$, respectively. Being an idempotent, $G$ satisfies $G^2 = G$ which gives us the set of equations $A = \{e^2+fg = e, f(e+h) = f, g(e+h) = g, fg+h^2 = h\}$. Consider  $(i) $  $e^2 +fg = e$, $(ii)$  $f(e+h) = f$, $(iii)$  $g(e+h) = g$ and $(iv)$  $fg+h^2 = h$.  From Theorems $2.1$ and $2.2$, we know that $\det  G$ is  also an idempotent of $\mathbb{Z}_{pqr}$. Thus, employ Lemma $3.1$ to obtain the   possible choices of $\det G$ given by $$\{ 0, \ 1,\ (pq)^{r-1},  \ (pr)^{q-1},\ (qr)^{p-1}, \ p^{(q-1)(r-1)},\ q^{(p-1)(r-1)},\ r^{(p-1)(q-1)}\}.$$  We consider all of  these  possibilities  in the following cases:

Case-1:  $\det G = 1$.  This means $G$ is invertible and as $G$ is an idempotent, on multiplying  the equation $G^2 = G$ by $G^{-1}$, we get $G = I$. Therefore, this case yields only trivial idempotent.

Case 2: $\det G = 0$. Let us first show that in this case, $e+h$ is an idempotent of $\mathbb{Z}_{pqr}$. $\det G = 0$. In this case, $eh-fg = 0$.    For this consider $(e+h)^2 = e^2+h^2+2eh = e^2+h^2 + 2fg$. Now add equations $(i)$ and $(iv)$, and using these  to deduce that $(e+h)^2 = e+h$, which means that $e+h$ is an idempotent. Therefore, it can take the following values: $$\{ 0, \ 1,\ (pq)^{r-1},  \ (pr)^{q-1},\ (qr)^{p-1}, \ p^{(q-1)(r-1)},\ q^{(p-1)(r-1)},\ r^{(p-1)(q-1)}\}.$$ Now we discuss all these possibilities one by one. 
\begin{enumerate}
\item If $e+h = 0$.  In this case, $\det  G = 0$ and $(i)$ implies $e = 0$. From $(ii)$ and $(iii)$, we get $f = g = 0$. Similarly $h = 0$. Hence, $G$ can only be $0$ matrix here. As we are in hunt of non trivial idempotents, we  reject this case.

\item If $e+h = 1$.  Clearly, $h = 1-e$. Employing $\det G = 0$ and trace condition, we can easily verify that  the equations $(i), (ii), (iii)$ are trivially satisfied. Thus, we have $G = \begin{bmatrix}
e(x) & f(x) \\ g(x) & 1-e(x)
\end{bmatrix}$, where $e(x)(1-e(x))-g(x)f(x) = 0$.

\item  If $e+h = (pq)^{r-1}$. Clearly $h = (pq)^{r-1}-e$. Employing $\det G = 0$ and $(i)$, i.e.  $e^2+fg = e$, we get $ e = e(pq)^{r-1}$.  Also from  $(ii)$ and $(iii)$, we have $f = f(pq)^{r-1}$ and $g = g(pq)^{r-1}$. Hence,  $G = \begin{bmatrix}
 e(x)(pq)^{r-1} & f(x)(pq)^{r-1} \\ g(x)(pq)^{r-1} & (pq)^{r-1}(1-e(x))
\end{bmatrix}$, where $e(x)(1-e(x))-g(x)f(x) =  r k(x)$ for some $k(x) \in \mathbb{Z}_{pqr}[x]$. Further, when $e+h = (qr)^{p-1}$ or $e+h = (pr)^{q-1}$, $G$ can be obtained in a  similar way.

\item If $e+h = p^{(q-1)(r-1)}$. Then $\det G = 0$ and $(i)$ implies $ep^{(q-1)(r-1)} = e$. Also from $(ii)$ and $(iii)$, we get $fp^{(q-1)(r-1)} = f$ and $gp^{(q-1)(r-1)} = g$. Thus in this case, $G = \begin{bmatrix}
 p^{(q-1)(r-1)}e(x) &  p^{(q-1)(r-1)} f(x) \\  p^{(q-1)(r-1)} g(x) &  p^{(q-1)(r-1)}(1-e(x))

\end{bmatrix}$, where $e(x)(1-e(x))-f(x)g(x) = pqk(x)$ for $k(x) \in \mathbb{Z}_{pqr}[x]$. Further,  when $e+h = q^{(p-1)(r-1)}$ or $e+h = r^{(p-1)(q-1)}$, $G$ can be obtained in a  similar way. 
\end{enumerate}

Case 3:  $\det G = (pq)^{r-1}$. Here $(i), (iv)$ and $\det G = (pq)^{r-1}$ implies $(e+h)^2 = e+h+ 2(pq)^{r-1}$. On incorporating Lemma $3.3$, we get $8$ possibilities for $e+h$ which are discussed in the following points:
\begin{enumerate}

\item If $e+h = 2(pq)^{r-1}$, then from equations $(ii)$ and $(iii)$, we get $2(pq)^{r-1}f = f$ and $ \ 2(pq)^{r-1}g = g$. Now as $\gcd(2(pq)^{r-1}-1, pqr) = 1$, we have $f = g = 0$. So equations $(i)$ and $(iv)$ implies $e^2 = e$ and $h^2 = h$. It can be easily seen that the only possible value of $e$ and $h$  in this case is $2(pq)^{r-1}$. Thus, $G = \begin{bmatrix}
(pq)^{r-1} & 0 \\ 0 & (pq)^{r-1}
\end{bmatrix}$.

\item If $e+h = -(pq)^{r-1}$, then from equations $(ii)$ and $(iii)$, we get $(1+(pq)^{r-1})f = 0$ and $ \ g(1+ (pq)^{r-1}) = 0$. Again as $\gcd(1+(pq)^{r-1}, pqr) = 1$, we have $f = g = 0$. So equations $(i)$ and $(iv)$ implies $e^2 = e$ and $h^2 = h$. It can be verified that  no two idempotents in $\mathbb{Z}_{pqr}$ have sum $-(pq)^{r-1}$. Thus, this case is not possible.

\item If $e+h = (pq)^{r-1}+1$, then $\det G = (pq)^{r-1}$ and $(i)$ implies $(pq)^{r-1} e = (pq)^{r-1}$. Also $(ii)$ and $(iii)$ implies $(pq)^{r-1} f = 0$ and $(pq)^{r-1}g = 0$. Thus from these equations, we get $G = \begin{bmatrix}
1+re(x) & rf(x) \\ rg(x) & (pq)^{r-1} -re(x)
\end{bmatrix}$, where $e(x)(1+re(x))+rf(x)g(x) = pq k(x)$ for some $k(x) \in  \mathbb{Z}_{pqr}[x]$. 

\item If $e+h = 1-2(pq)^{r-1}$, then $\det G = (pq)^{r-1}$ and $(i)$ implies $-2(pq)^{r-1} e = (pq)^{r-1}$. From equations $(ii)$ and $(iii)$, we get $2(pq)^{r-1})f = 0$ and $ \ 2(pq)^{r-1}g  = 0$. Multiply by $(2(pq)^{r-1})^2$ on both sides of $e^2+fg = e$ , we get $(2(pq)^{r-1}e)^2 + (2(pq)^{r-1}f)(2(pq)^{r-1}g) = (2(pq)^{r-1})^2e$. Using above equations, we get $$\big((pq)^{r-1}\big)^2 -2(pq)^{r-1}(-(pq)^{r-1}) = 3(pq)^{2(r-1)} = 0.$$ But, this is not possible by Euler's theorem   because of the choice of primes.

\item If $e+h = (2-p^{q-1})(pq)^{r-1}+p^{q-1}$, then $\det G = (pq)^{r-1}$ and $(i)$ implies $$e\big((2-p^{q-1})(pq)^{r-1}+p^{q-1}-1\big) = (pq)^{r-1}.$$ We rewrite above as $e(t-1) = (pq)^{r-1}$, where $t = e+h$. From this equation, we conclude that $e$ is of the form $u+pre$ for some $u$ such that $u \equiv 0 \Mod p$ and $u \equiv 1 \Mod r$. 
Also $(ii)$ and $(iii)$ implies $ f(t-1)= 0$ and $g(t-1) = 0$. Thus, on solving these equations, we get $G = \begin{bmatrix}
u+pre(x) & prf(x) \\ pr g(x) & t-(u+pre(x))
\end{bmatrix}$, where $(u+pre(x))(t-(u+pre(x))) -(pr)^2f(x)g(x) \equiv (pq)^{r-1}.$

\item If $e+h = (-1-p^{q-1})(pq)^{r-1}+p^{q-1} $. Let $(-1-p^{q-1})(pq)^{r-1}+p^{q-1} = t$.  Then  $\det G = (pq)^{r-1}$ and $(i)$ implies $e(t-1) = (pq)^{r-1}$. Also $(ii)$ and $(iii)$ implies $f(t-1) = 0$ and $g(t-1) = 0$. Now on multiplying by $(t-1)^2$ on both sides of $(i)$, we get $$ (e(t-1))^2+f(t-1)g(t-1) = (t-1)(e(t-1)).$$ This further implies that $$((pq)^{r-1})^2 = (t-1)(pq)^{r-1}.$$ But under modulo $r$ above expression does not hold, as left side is congruent to $1$ and right side is congruent to $-2$. Hence, this case is not possible.

\item  If $e+h = (2-q^{p-1})(pq)^{r-1}+q^{p-1}$. This case is exactly similar to sub-case $5$. Just replace $p$ and $q$. 
\item If $e+h = (-1-q^{p-1})(pq)^{r-1}+q^{p-1} $. Similar to sub-case $6$, we can prove that this case is not possible.

\end{enumerate}
For $\det G=(qr)^{p-1}$ or $(pr)^{q-1}$, idempotents can be determined with the similar approach as applied for $\det G=(pq)^{p-1}$.

 Case 4:  $\det G = p^{(q-1)(r-1)}$.  In this case,  $(i), (iv)$ and $\det G = p^{(q-1)(r-1)}$ implies $(e+h)^2 = e+h+ 2p^{(q-1)(r-1)}$. On employing Lemma $3.2$, we get $8$ possibilities for $e+h$ and therefore, we discuss all these possibilities. 

\begin{enumerate}

\item If $e+h = 2p^{(q-1)(r-1)}$, then from equations $(ii)$ and $(iii)$, we get $2 p^{(q-1)(r-1)}f = f$ and $ \ 2p^{(q-1)(r-1)} g = g$. Now as $\gcd(2p^{(q-1)(r-1)}-1, pqr) = 1$, we get $f = g = 0$. So equations $(i)$ and $(iv)$ implies $e^2 = e$ and $h^2 = h$. It can be easily verified that the only possible value of $e$ and $h$ in this case is $p^{(q-1)(r-1)}$. Thus, $G = \begin{bmatrix}
p^{(q-1)(r-1)} & 0 \\ 0 & p^{(q-1)(r-1)}
\end{bmatrix}$.

\item If $e+h = 1+p^{(q-1)(r-1)}$, then $\det G = p^{(q-1)(r-1)}$ and $(i)$ implies $(p^{(q-1)(r-1)}) e = p^{(q-1)(r-1)}$. Also $(ii)$ and $(iii)$ implies $p^{(q-1)(r-1)} f = 0$ and $p^{(q-1)(r-1)}g = 0$. Thus, from these equations, we get $G = \begin{bmatrix}
1+qre(x) & qrf(x) \\ qrg(x) & p^{(q-1)(r-1)}-qre(x)
\end{bmatrix}$, where $e(x)(1+qre(x))+qrf(x)g(x) = pk(x)$ for some $k(x) \in  \mathbb{Z}_{pqr}[x]$. 
\item If $e+h = -p^{(q-1)(r-1)}$, then from equations $(ii)$ and $(iii)$, we get $(p^{(q-1)(r-1)}+1)f = 0$ and $ \ (p^{(q-1)(r-1)}+1)g = 0$. Now as $\gcd(p^{(q-1)(r-1)}+1, pqr) = 1$, we have $f = g = 0$. So equations $(i)$ and $(iv)$ implies $e^2 = e$ and $h^2 = h$. It can be easily verified that there are no two idempotents in $\mathbb{Z}_{pqr}$ whose sum is $-p^{(q-1)(r-1)}$. 
\item If $e+h = 1-2p^{(q-1)(r-1)}$, then $\det G = p^{(q-1)(r-1)}$ and $(i)$ implies $(2p^{(q-1)(r-1)}) e = -p^{(q-1)(r-1)}$. Also $(ii)$ and $(iii)$ implies $2p^{(q-1)(r-1)} f = 0$ and $2p^{(q-1)(r-1)}g = 0$. Now on multiplying by $(2p^{(q-1)(r-1)})^2$ on both sides of $(i)$, i.e. $e^2+fg = e$, we get $$(2p^{(q-1)(r-1)}e)^2 + (2p^{(q-1)(r-1)}f)(2p^{(q-1)(r-1)}g) = 2p^{(q-1)(r-1)} (2p^{(q-1)(r-1)}e).$$ On substituting values in above equation, we get $$(-p^{(q-1)(r-1)})^2 = 2p^{(q-1)(r-1)}(-p^{(q-1)(r-1)}) \implies 3p^{2(q-1)(r-1)} = 0.$$  But above is not possible by Euler's theorem as $p , q, r$ are greater than $3$. 
\end{enumerate}For the remaining values of trace, i.e.  $$\big\{((-1-2p^{q-1}\big)(pq)^{r-1} + 2 p^{q-1}),  \ \big((-2-p^{q-1})(pq)^{r-1}+p^{q-1}+1\big), $$ $$ \ ((-1-2p^{r-1})(pr)^{q-1}+2p^{r-1}), \ \big((-2-p^{r-1})(pr)^{q-1}+p^{r-1}+1\big)\big\},$$ we can prove  that these cases are not possible on the similar lines of sub-case $(4)$ above. So, there are only $2$ classes of idempotents having determinant $p^{(q-1)(r-1)}$. Further, the cases when $\det G = q^{(p-1)(r-1)}$ or $r^{(p-1)(q-1)}$ can be handled similarly.

Finally, on the other hand, one can see that all the matrices given in the statement of the theorem are idempotents. Thus, result.
\section{Discussion}
The study of idempotents is very important from the application point of view. In this paper, we have obtained all the possible idempotents of the matrix ring $M_2(\mathbb{Z}_{pqr}[x])$. This paper further motivates the study of non trivial idempotents of the matrix ring  $M_2(\mathbb{Z}_{n}[x])$, where $n$ is a square free integer having atleast $4$  prime factors.\end{proof}

\bibliographystyle{plain}

\begin{thebibliography}{10}
\small{

\bibitem {}{J. M. P. Balmaceda and J.  P. Datu,} \newblock{Idempotents in certain matrix rings  over polynomial rings, Int. Electron. J. Algebra, $\bf{27}$ ($2020$), $1-12$}.

\bibitem {}{ K. R. Goodearl,} \newblock{Von Neumann Regular Rings, Krieger Pub. Co. ($1991$)}.





\bibitem {}{P. Kanwar, M. Khatkar and R. K. Sharma,} \newblock{Idempotents and units of matrix rings  over polynomial rings, Int. Electron. J. Algebra, $\bf{22}$ ($2017$) $147-169$}.



\bibitem {}{ P. Kanwar, A. Leroy, and J. Matczuk,} \newblock{Idempotents in ring extensions, J. Algebra
$\bf{389}$ $(2013)$, $128-136$}.

\bibitem {}{P. Kanwar, A. Leroy and J. Matczuk, } \newblock{Clean elements in polynomial rings,}  \newblock{Noncommutative Rings and their Applications, Contemp. Math., Amer. Math.
Soc., $\bf{634}$ $(2015)$, $197-204$.}



\bibitem {}{W. K. Nicholson,} \newblock{Lie Lifting idempotents and exchange rings, Trans. Amer. Math.
Soc., $\bf{229}$  $(1977)$,  $ 269-278$}.

\bibitem {}{W. K. Nicholson,} \newblock{Strongly clean rings and Fitting's lemma, Comm. Algebra,  $\bf{27}(8)$  $(1999)$,  $ 3583-3592$}.

 \bibitem {}{R. K. Sharma, P. Yadav, and P. Kanwar,} \newblock{Lie regular generators of general
linear groups, Comm. Algebra, $\bf{40}$  $(2012)$,  $ 1304-1315$}.      
 

 
\small}






 \end{thebibliography}

\end{document}